\documentclass[11pt]{amsart}

\usepackage{amsmath}
\usepackage{amssymb}
\usepackage{amsfonts}
\usepackage{amsthm}
\usepackage{enumerate}
\usepackage{bbm}
\usepackage[mathscr]{eucal}
\usepackage{graphicx}
\usepackage{verbatim}
\usepackage{dsfont}
\usepackage{hyperref}
\hypersetup{colorlinks=true, urlcolor=black, linkcolor=black, citecolor=black}
\allowdisplaybreaks
\numberwithin{equation}{section}

\newtheorem{theorem}{Theorem}
\newtheorem{lemma}{Lemma}

\begin{document}
\title[Sets of Salem type]{Sets of Salem type and sharpness of the \\ $L^2$-Fourier restriction theorem}
\author{Xianghong Chen}
\address{X. Chen\\Department of Mathematics\\University of Wisconsin-Madison\\Madison, WI 53706, USA}
\curraddr{}
\email{xchen@math.wisc.edu}
\thanks{This research was supported in parts by NSF grants 0652890 and 1200261}
\dedicatory{}
\commby{}

\begin{abstract}
We construct Salem sets on the real line with endpoint Fourier decay and near-endpoint regularity properties. This complements a result of \L aba and Pramanik, who obtained near-endpoint Fourier decay and endpoint regularity properties. We then modify the construction to extend a theorem of Hambrook and \L aba to show sharpness of the $L^2$-Fourier restriction estimate by Mockenhaupt and Bak-Seeger, including the case where the Hausdorff and Fourier dimension do not coincide.
\end{abstract}
\maketitle

\section{Introduction and statement of results}
A compact set $E\subset\mathbb R^d$ of Hausdorff dimension $\alpha\in(0,d)$ is called a {\it Salem set} if there exists a Borel probability measures $\mu$ supported on $E$ satisfying the Fourier decay property
\begin{equation}\label{eq:fourier-decay-intro}
|\hat\mu(\xi)|\lesssim_\beta |\xi|^{-\beta/2}, \ \text{as } |\xi|\rightarrow\infty
\end{equation}
for all $\beta<\alpha$. It is known that $\beta$ must satisfy $\beta\le\alpha$ given the dimension constraint (see \cite{salem}; see also \cite[Corollary~8.7]{wolff}). In $d\ge 2$, a typical example of Salem set is the sphere $S^{d-1}$ for which \eqref{eq:fourier-decay-intro} holds with $\mu$ being the (normalized) surface measure and $\beta=\alpha=d-1$.

The existence of Salem sets in $d=1$ was first shown by Salem \cite{salem} (see also Bluhm \cite{bluhm} for a variant of Salem's construction). Different constructions were found later with more quantitative estimates near $\beta=\alpha$. See e.g. a construction of Kahane \cite{kahane} using Brownian motion, and a construction of Kaufman \cite{kaufman} using the prime number theorem. K\"orner \cite{koerner1978} constructed, using a bootstrapping argument, Salem sets of zero $\alpha$-dimensional Hausdorff measure that attain the endpoint Fourier decay rate. An argument using Baire category can be found in K\"orner \cite{koerner2011} where the constructed measures satisfy instead the upper regularity property
\begin{equation}\label{eq:ahlfors-intro}
\mu\big(B(x,r)\big)\lesssim r^\alpha,\ \ \forall x\in\mathbb R^d, r>0.
\end{equation}

In view of the spheres it is natural to ask for stronger regularity property
\begin{equation}\label{eq:AD-intro}
\mu\big(B(x,r)\big)\approx r^\alpha,\ \ \forall x\in E, 0<r<1
\end{equation}
for Salem sets to satisfy. Note that \eqref{eq:AD-intro} implies that $E$ is of Hausdorff and Minkowski dimension $\alpha$ (see e.g. \cite[Theorem~5.7]{mattila}). However, to the best of our knowledge, none of the above-mentioned constructions exactly verifies \eqref{eq:AD-intro}. Indeed, Mitsis \cite{mitsis} and Mattila \cite[Problem~20]{mattila} have asked whether there exist Salem sets in $d=1$ satisfying \eqref{eq:AD-intro} and \eqref{eq:fourier-decay-intro} with $\beta=\alpha$.

Progress along this line has been made by \L aba and Pramanik \cite[Section~6]{laba-pramanik} who construct measures satisfying \eqref{eq:AD-intro} and \eqref{eq:fourier-decay-intro} for all $\beta<\alpha$ (in fact more quantitative estimate is available near $\beta=\alpha$, see Theorem \ref{thm:theorem-3} below). Our first result complements theirs by constructing measures satisfying \eqref{eq:fourier-decay-intro} with $\beta=\alpha$ and a variant of \eqref{eq:AD-intro} with slightly higher dimensionality.

\begin{theorem}\label{thm:theorem-2}
Given $0<\alpha<1$, there exists a Borel probability measure $\mu$ on $\mathbb R$ supported on a compact set $E$ such that
\begin{equation}\label{eq:fourier-decay-2}
|\hat\mu(\xi)|\lesssim|\xi|^{-\alpha/2}
\end{equation}
for all $\xi\in\mathbb R, \xi\neq 0$, and
\begin{equation}\label{eq:regularity-2}
\mu(I)\approx\frac{|I|^\alpha} {\log(1/|I|)}
\end{equation}
for all intervals $I$ centered in $E$ with $|I|<1/2$.
\end{theorem}

The following result is implicit in \cite[Section~6]{laba-pramanik} for dimensions of the form $\alpha = \log_N M$ (where $1 < M < N$ are integers); we generalize it to all $\alpha\in(0,1)$.

\begin{theorem}\label{thm:theorem-3}
Given $0<\alpha<1$, there exists a Borel probability measure $\mu$ on $\mathbb R$ supported on a compact set $E$ such that
\begin{equation}\label{eq:fourier-decay-3}
|\hat\mu(\xi)|\lesssim|\xi|^{-\alpha/2} {\log^{1/2}|\xi|}
\end{equation}
for all $\xi\in\mathbb R, |\xi|\ge 2$, and
\begin{equation}\label{eq:regularity-3}
\mu(I)\approx {|I|^\alpha}
\end{equation}
for all intervals $I$ centered in $E$ with $|I|<1$.
\end{theorem}

\vspace{0.1cm}
It is well known that the Fourier decay and regularity properties of $\mu$ are related to the associated $L^2$-Fourier extension estimate
\begin{equation}\label{eq:fourier-extension-intro}
\|\widehat {f\mu}\|_{L^q(\mathbb R^d)}\lesssim \|f\|_{L^2(\mu)}.
\end{equation}
Here the implied constant is independent of the bounded Borel function $f$. By duality \eqref{eq:fourier-extension-intro} is equivalent to the $L^2$-Fourier restriction estimate
$$\|\hat g\|_{L^2(\mu)}\lesssim \|g\|_{L^{q'}(\mathbb R^d)}$$
where $q'$ is the conjugate exponent of $q$. It is known that $\|\hat\mu\|_q=\infty$ for all $q<2d/\alpha$ (see \cite{salem}; see also \cite[Corollary~8.7]{wolff}). Thus $q$ must satisfy $q\ge 2d/\alpha$. When $\mu$ is the surface measure on the sphere, the Stein-Tomas theorem \cite{tomas}, \cite{stein} states that \eqref{eq:fourier-extension-intro} holds for $\frac{2(d+1)}{d-1}\le q\le \infty$. The sharpness of this range is justified by Knapp's homogeneity argument which captures the ``curvedness'' of the sphere. It was shown by Mockenhaupt \cite{mockenhaupt} (see also Mitsis \cite{mitsis}) that Tomas's argument in \cite{tomas} can be used to show that \eqref{eq:fourier-extension-intro} holds for $$2+\frac{4(d-\alpha)}{\beta}=:q_0(\alpha,\beta)<q\le\infty$$
provided that $\mu$ satisfies \eqref{eq:fourier-decay-intro} and \eqref{eq:ahlfors-intro}. Bak and Seeger \cite{bak-seeger} extend this result to the endpoint $q_0$ and further strengthen \eqref{eq:fourier-extension-intro} by replacing $L^q$ with the Lorentz space $L^{q,2}$.

However, for general measures satisfying \eqref{eq:fourier-decay-intro} and \eqref{eq:ahlfors-intro}, there is no Knapp argument available to prove sharpness of $q_0$.\footnote{Indeed, using a general result from \cite{chen}, it is shown in \cite{chen-seeger} that for $\alpha=d/n$, $n=2,3,\cdots$, the range can be extended to ${2d}/{\alpha}\le q\le \infty$ for some $\alpha$-dimensional measures satisfying \eqref{eq:ahlfors-intro} and \eqref{eq:fourier-decay-intro} with $\beta=\alpha$.} Remarkably, by building arithmetic progressions into the Salem sets constructed by \L aba and Pramanik \cite{laba-pramanik}, Hambrook and \L aba \cite{hambrook-laba} show that the dependence of $q_0=q_0(\alpha,\alpha)$ on $\alpha$ is not improvable. More precisely, for any dimension of the form $\alpha={\log_N M}$, they construct a compact set $E\subset\mathbb R$ of dimension $\alpha$ and a Borel probability measure $\mu$ supported on $E$ satisfying \eqref{eq:ahlfors-intro} and \eqref{eq:fourier-decay-intro} for any $\beta<\alpha$; moreover, given $2\le q<q_0(\alpha,\alpha)$, they find such a measure $\mu$ and a sequence of indicator functions (of finite union of intervals) with $\|f_l\|_{L^2(\mu)}>0$ and
$$\sup_{l\ge 1}\frac{\|\widehat{f_l\mu}\big\|_{L^{q}(\mathbb R)}}{\|f_l\|_{L^2(\mu)}}=\infty.$$
Since the constructed measures depend on $q$, their result only shows sharpness of $q_0=q_0(\alpha,\alpha)$ at the level of its dependence on $\alpha$; that is, for any fixed example in their construction, the full range of $q$ for which \eqref{eq:fourier-extension-intro} holds is not known (even up to endpoint). Also, since \eqref{eq:fourier-decay-intro} is verified only for $\beta<\alpha$, the result of Bak and Seeger does not apply to $q_0(\alpha,\alpha)$.

Here, by modifying their construction, we find Salem sets for which \eqref{eq:fourier-decay-intro} is satisfied with $\beta=\alpha$ and the range of $q$ for which \eqref{eq:fourier-extension-intro} holds is precisely $q_0(\alpha,\alpha)\le q\le\infty$. We further extend this result to the case $q_0=q_0(\alpha,\beta)$, thereby showing the sharpness of Bak-Seeger's estimate in greater generality.

\begin{theorem}\label{thm:theorem-1}
Given $0<\beta\le\alpha<1$ and a nondecreasing function $\phi: [2,\infty)\rightarrow (0,\infty)$ satisfying $\lim_{t\rightarrow\infty}\phi(t)=\infty$ and $\phi(2t)\le \phi(t)+C$, there exist a Borel probability measure $\mu$ on $\mathbb R$ supported on a compact set $E$ and a sequence of indicator functions (of finite union of intervals) $\{f_l\}_{l\ge 1}$ with $\|f_l\|_{L^2(\mu)}>0$ so that the following estimates hold.
\begin{equation}\label{eq:fourier-decay}
|\hat\mu(\xi)|\lesssim|\xi|^{-\beta/2}
\end{equation}
for all $\xi\in\mathbb R, \xi\neq 0$.
\begin{equation}\label{eq:regularity}
\frac{|I|^\alpha} {\phi(1/|I|)\log(1/|I|)}\lesssim \mu(I)
\lesssim \frac{|I|^\alpha} {\log(1/|I|)}
\end{equation}
for all intervals $I$ centered in $E$ with $|I|<1/2$.
\begin{equation}\label{eq:sharpness}
\sup_{l\ge 1}\frac{\|\widehat{f_l\mu}\big\|_{L^{q}(\mathbb R)}}{\|f_l\|_{L^2(\mu)}}=\infty.
\end{equation}
for all $2\le q<2+\frac{4(1-\alpha)}{\beta}$.

Moreover, in the case $\beta<\alpha$, there exists such a measure $\mu$ satisfying
\begin{equation}\label{eq:regularity-4}
\frac{|I|^\alpha} {\phi(1/|I|)}\lesssim \mu(I)
\lesssim {|I|^\alpha}
\end{equation}
in place of \eqref{eq:regularity}, for all intervals $I$ centered in $E$ with $|I|<1/2$.
\end{theorem}

One can take, for example, $\phi(t)=\log^{\epsilon}t$. In this case both \eqref{eq:regularity} and \eqref{eq:regularity-4} imply
$$|I|^{\alpha+\delta}\lesssim_{\delta} \mu(I)\lesssim |I|^\alpha$$
for all $\delta>0$ and all intervals $I$ centered in $E$ with $|I|<1/2$. In particular, $E$ is of Hausdorff and Minkowski dimension $\alpha$.

Our construction differs from that of \L aba-Pramanik \cite{laba-pramanik} and Hambrook-\L aba \cite{hambrook-laba} in two main aspects. First, to obtain the Fourier decay property \eqref{eq:fourier-decay-intro} with $\beta=\alpha$, we increase the dimensionality of the Salem set by a logarithmic scale. This causes some difficulty as the logarithmic factor does not fit directly into the $N^\alpha$-out-of-$N$ Cantor-type construction utilized in \cite{laba-pramanik}. Our main innovation is to replace such a construction with a deficient-or-excessive-out-of-$N$ construction which turns out to be rather convenient and flexible. Second, to obtain sharpness of the Fourier extension estimate \eqref{eq:fourier-extension-intro} in exactly the range $q_0(\alpha,\beta)\le q\le\infty$, instead of keeping $N$ constant we let $N$ grow slowly to infinity (quantified by $\phi$) in order to build in arithmetic progressions of growing length. This comes at the expense of a slight weakening of the regularity estimates, as shown in \eqref{eq:regularity} and \eqref{eq:regularity-4}.

\vspace{0.1cm}
The plan of the paper is as follows. In Section \ref{sec:theorem-2} we present the proof of Theorem \ref{thm:theorem-2} which can also serve as a good warm-up for the proof of Theorem \ref{thm:theorem-1}. In Section \ref{sec:theorem-3} we prove Theorem \ref{thm:theorem-3} which follows quickly from the proof of Theorem \ref{thm:theorem-2} by adjusting parameters. In Section \ref{sec:theorem-1} we prove Theorem \ref{thm:theorem-1} following closely the argument in \cite{hambrook-laba} with some simplifications.

\vspace{0.1cm}
\emph{Notation.} We write $\square_1\lesssim \square_2$ to indicate that $\square_1\le C\square_2$ for some constant $0<C<\infty$ independent of the testing inputs which will usually be clear from the context. By $\square_1\approx\square_2$ we mean both $\square_1\lesssim\square_2$ and $\square_2\lesssim\square_1$ hold. For a finite Borel measure $\mu$ on $\mathbb R^d$, we write for $\xi\in\mathbb R^d$
$$\hat\mu(\xi)=\mathcal F(\mu)(\xi)=\int_{\mathbb R^d}e^{-2\pi i\langle\xi,x\rangle}d\mu(x)$$
where $\langle\xi, x\rangle$ is the Euclidean inner product. We say that a Borel probability measure $\mu$ is supported on a set $E$ if $\mu(E)=1$. We denote by $[N]$ the set $\{0,\cdots,N-1\}\subset \mathbb R$ and $[N]/n=\{0,1/n,\cdots,(N-1)/n\}$. $\log(\cdot)$ denotes the logarithmic function with base $e$.
\section{Proof of Theorem \ref{thm:theorem-2}}\label{sec:theorem-2}

\subsection{Outline of the construction}
Given $0<\alpha<1$ and a sequence of ``branching numbers'' (to be chosen in Section \ref{sec:choosing-2})
$$\{t_N\}_{N\ge 1} \text{ with } t_N\in \{1,2\},$$
we will construct inductively a sequence of ``dyadic nodes'' (i.e. subsets of the dyadic rationals) $A_0, A_1, \cdots$ on $[0,1]$ with
\begin{align*}
&\bullet\ \  A_0=\{0\}\\
&\bullet\ \  A_N=\bigcup_{a\in A_{N-1}}A_{N,a}+a\\
&\bullet\ \  A_{N,a}\subset \{0,1\}/2^N\\
&\bullet\ \  \# A_{N,a}=t_N.
\end{align*}
Note that
$$A_N\subset \{0,1,\cdots,2^N-1\}/2^N.$$

Our main observation is that, although the dimension $\alpha$ is not reflected at every single stage of branching, the right dimensionality can still be achieved as long as
$$T_N:=\#A_N=t_1\cdots t_N$$
has the right magnitude (see Section \ref{sec:choosing-2} for precise formulation).

Given the set $A_N$, we write the union of the associated intervals
$$E_N=\bigcup_{a\in A_{N}}[0,1/2^N]+a$$
and the associated uniform probability measure
$$\mu_N=\frac{1}{|E_N|}\mathds{1}_{E_N}(t)dt.$$
And let
$$E=\bigcap_{N=1}^\infty E_N$$
be the Cantor set on which the Cantor measure
$$\mu=\lim_{N\rightarrow\infty} \mu_N$$
is supported, where the limit is taken in the weak sense (the existence of the weak limit is standard, see Section \ref{sec:limit} below).

\subsection{Existence of the weak limit}\label{sec:limit}
Denote by
$$F_N(t)=\mu_N(-\infty,t]$$
the cumulative distribution function of $\mu_N$. Then by inspection we have
$$\|F_{N+1}-F_N\|_\infty\le T_N^{-1}.$$
By our choice of $\{t_N\}_{N\ge 1}$ in Section \ref{sec:choosing-2},
$$T_N\approx 2^{N\alpha}N.$$
Thus the continuous functions $F_N$ converge uniformly to a cumulative distribution function of a probability measure $\mu$, which implies that $\mu_N$ converges weakly to $\mu$. See e.g. \cite[Section~3.2]{durrett}.

\subsection{The iteration}\label{sec:iteration}
Given a large integer $N_0$ (to be chosen in Section \ref{sec:choosing-2}), we choose $A_1, \cdots,A_{N_0}$ arbitrarily as long as $E_{N_0}$ is compactly supported in $(0,1)$. Assuming that $A_1,\cdots, A_{N-1}$ have been chosen, we now choose $A_{N}$.

Let $A\subset\mathbb R$ be a finite set, we will write for $k\in \mathbb Z$
$$S_A(k)=\sum_{a\in A}e^{-2\pi i ak}.$$
The Fourier decay property \eqref{eq:fourier-decay-2} will follow if we have an efficient bound on $|\hat\mu_N(k)-\hat\mu_{N-1}(k)|$. This reduces to bounding a ``sparse sum'' of the form
$$\sum_{a\in A_{N-1}} \Big(\frac{S_{A_N,a}(k)}{t_N}-\frac{S_{\{0,1\}/2^N}(k)}{2}\Big) e^{-2\pi iak}$$
where $k=1,\cdots,2^N-1$. When $t_N=2$ this is trivial since each term is equal to $0$. When $t_N=1$, we need to choose $A_{N,a}=\{0\}/2^N$ or $\{1\}/2^N$ for each $a\in A_{N-1}$ to obtain a favorable bound. To this end we use the random selection employed in \cite{laba-pramanik}.

More precisely, let
$$B_N=\{0,1,\cdots,t_N-1\}/2^N.$$
We write for $x\in \{0,1\}$
$$B_{N,x}=\big\{{y+x\ (\text{mod } 2)}: y/2^N\in B_N\big\}/{2^N}\subset \{0,1\}/2^N,$$
the translation of $B_N$ by $x$ within $\{0,1\}/2^N$. Clearly,
$$\Big|\frac{S_{B_N,x}(k)}{t_N}-\frac{S_{\{0,1\}/2^N}(k)}{2}\Big|\le 2.$$
For fixed $k\in\mathbb Z$, consider the random variables
$$\chi_a(k)=\Big(\frac{S_{B_N,x(a)}(k)}{t_N}-\frac{S_{\{0,1\}/2^N}(k)}{2}\Big) e^{-2\pi iak},\ a\in A_{N-1}$$
where $\{x(a)\}_{a\in A_{N-1}}$ are independent random variables uniformly distributed on $\{0,1\}$.\footnote{Here we are omitting the explicit construction of the probability space. One can take, for example, the sample space $\Omega=\prod_{a\in A_{N-1}}\{0,1\}$ equipped with the uniform probability measure, and the random variables $x(a)=\pi_a$ (the projection to the $a$-th coordinate).} To bound the sum $\sum_{a\in A_{N-1}}\chi_a(k)$ we invoke Bernstein's inequality. For a proof, see e.g. \cite{hoeffding}.

\begin{lemma}[Bernstein's inequality]\label{lem:bernstein}
Let $X_1,\cdots,X_n$ be independent complex-valued random variables with $\mathbb EX_j=0$ and $|X_j|\le M\in (0,\infty)$ for all $j=1,\cdots,n$. Then for $0<t\le 1$,
$$\mathbb P\Big(\Big|\frac{1}{n}\sum_{j=1}^n X_j\Big|\ge Mt \Big)\le 4 e^{-{nt^2}/{4}}.$$
\end{lemma}

Notice that $\{\chi_a(k)\}_{a\in A_{N-1}}$ are independent random variables with
$$\mathbb E\chi_a(k)=0$$
and  $|\chi_a(k)|\le 2$. So we can choose
$$t^2=T_{N-1}^{-1}4\log(2^N 8)$$
in Lemma \ref{lem:bernstein} to obtain
$$\mathbb P
\Big(\Big|\frac{1}{T_{N-1}}\sum_{a\in A_{N-1}}\chi_a(k)\Big|
\ge {4}{T_{N-1}^{-1/2}}\log^{1/2}(2^N 8)
\Big)
\le \frac{1}{2^N2},$$
provided that
\begin{equation}\label{eq:bernstein-cond-2}
4\log(2^N 8)\le T_{N-1}
\end{equation}
which will be guaranteed by our choice of $t_N$ and $N_0$ (see Section \ref{sec:choosing-2}). Allowing $k$ to vary, we see that we can choose $x(a)$ in such a way that
\begin{align}\label{eq:key-est-pre-2}
\Big|\frac{1}{T_{N-1}}\sum_{a\in A_{N-1}}\chi_a(k)\Big|
\le 4{T_{N-1}^{-1/2}}\log^{1/2}(2^N8)
\end{align}
for $k=0,1,\cdots, 2^N-1$. By periodicity, this bound extends to all $k\in\mathbb Z$.

Now we set $A_{N,a}=B_{N,x(a)}$ and
$$A_N=\bigcup_{a\in A_{N-1}}A_{N,a}+a.$$
This finishes our choice of $A_N$.

\subsection{{Choosing $t_N$ and $N_0$}}\label{sec:choosing-2}
We now choose $t_N$ and $N_0$. The way we choose $t_N$ is the key ingredient of our construction. The same idea will be used in the proof of Theorem \ref{thm:theorem-3} and Theorem \ref{thm:theorem-1}. From now on till the end of Section \ref{sec:theorem-3}, all implied constants depend only on $\alpha$.

In view of \eqref{eq:key-est-pre-2}, we want
\begin{align}\label{want1-2}
{T_{N-1}^{-1/2}}\log^{1/2}(2^N 8)\approx {2^{-N\alpha/2}}
\end{align}
for all $N>N_0$. If we write
$$t_N=2^{\alpha}\theta_N,$$
then \eqref{want1-2} reduces to
\begin{align}\label{eq:asympt-2}
\log(2^N8)\approx \theta_1\cdots\theta_N.
\end{align}
Since $\theta_N$ is either $2^{-\alpha}$ or $2^{1-\alpha}$, this can be guaranteed by choosing appropriate sequence $t_N\in\{1,2\}$. This finishes our choice of $t_N$. One can then choose $N_0$ large enough so that \eqref{eq:bernstein-cond-2} is satisfied.

\subsection{Fourier decay estimate}\label{sec:fourier-decay-2}
We can now prove \eqref{eq:fourier-decay-2}. By \eqref{eq:key-est-pre-2} and \eqref{want1-2},
$$\Big|\frac{1}{T_{N-1}}\sum_{a\in A_{N-1}}\chi_a(k)\Big|\lesssim {2^{-N\alpha/2}}$$
holds for all $N>N_0$ and $k\in\mathbb Z$. Consequently, for $k\neq 0$,
\begin{align*}
|\hat\mu_N(k)-\hat\mu_{N-1}(k)|
&= \Big|\mathcal F\big(2^N\mathds{1}_{[0,1/2^N)}\big)(k)\Big|\Big|\frac{1}{T_{N-1}}\sum_{a\in A_{N-1}}\chi_a(k)\Big|\\
& \lesssim \min\Big(1,\frac{2^N}{|k|}\Big)\Big|\frac{1}{T_{N-1}}\sum_{a\in A_{N-1}}\chi_a(k)\Big|\\
& \lesssim \min\Big(1,\frac{2^N}{|k|}\Big)\frac{1}{2^{N\alpha/2}}.
\end{align*}
To bound the sum in $N$, we write
\begin{align*}
\sum_{N=1}^\infty\min\Big(1,\frac{2^N}{|k|}\Big)\frac{1}{2^{N\alpha/2}}
&= \sum_{2^N\le |k|}\frac{2^N}{|k|}\frac{1}{2^{N\alpha/2}}+\sum_{2^N>|k|}\frac{1}{2^{N\alpha/2}}\\
&= \frac{1}{|k|}\sum_{N\le N_1}2^{N(1-\alpha/2)}+\sum_{N\ge N_2}{2^{-N\alpha/2}}
\end{align*}
where $N_1$ is the largest integer such that $2^{N_1}\le |k|$ and $N_2$ is the smallest integer such that $2^{N_2}> |k|$. Since
$$\sum_{N\le N_1}2^{N(1-\alpha/2)}\lesssim 2^{N_1(1-\alpha/2)}$$
and
$$\sum_{N\ge N_2} {2^{-N\alpha/2}}\lesssim {2^{-N_2\alpha/2}},$$
we get
\begin{align*}
\sum_{N=1}^\infty\min\Big(1,\frac{2^N}{|k|}\Big)\frac{1}{2^{N\alpha/2}}
&\lesssim {|k|^{-\alpha/2}}\frac{2^{N_1(1-\alpha/2)}}{|k|^{1-\alpha/2}}+{2^{-N_2\alpha/2}}\\
&\lesssim {|k|^{-\alpha/2}}.
\end{align*}
Now notice that, by the weak convergence of $\mu_N$, we have
$$\hat\mu(k)=\lim_{N\rightarrow\infty}\hat\mu_N(k).$$
Also, it is easy to see that
$$|\hat\mu_N(k)|\lesssim {|k|^{-1}}$$
for $N=1,\cdots,N_0$. So we can conclude, by the estimates above,
\begin{align*}
|\hat\mu(k)|
&\le |\hat\mu_1(k)|+\sum_{N=2}^\infty |\hat\mu_N(k)-\hat\mu_{N-1}(k)|
\lesssim {|k|^{-\alpha/2}}
\end{align*}
for all $k\in\mathbb Z, k\neq 0$. The estimate \eqref{eq:fourier-decay-2} for general $\xi$ then follows from the fact that $\mu$ is compactly supported in $(0,1)$. See e.g. \cite[p.~252]{kahane}.

\subsection{Regularity estimate}\label{sec:regularity-2}
To finish the proof of Theorem \ref{thm:theorem-2}, it remains to show \eqref{eq:regularity-2}. We split the proof into two parts.

\begin{lemma}\label{lem:upper-regular-2}
For all intervals $I$ with $|I|<1/2$,
$$\mu(I)\lesssim \frac{|I|^\alpha}{\log(1/|I|)}.$$
\end{lemma}

\begin{proof}
It suffices to prove the estimate for small intervals. Assume that
$$\frac{1}{2^{N+1}}\le |I|<\frac{1}{2^N}$$
and that there exists $a\in A_{N}$ such that the interval $(0,1/2^N)+a$ intersects with $I$. Since $|I|<1/2^N$, there are at most two such intervals, say $J_1$ and $J_2$. Now we can estimate
\begin{align*}
\mu(I)
&=\mu(I\cap J_1)+\mu(I\cap J_2)\ \text{ (since $\mu$ is supported on $E_N$)}\\
&\le \mu(J_1)+\mu(J_2)\\
&={2}{(t_1\cdots t_N)^{-1}}\  \hspace{1.34cm}\text{ (by the definition of $\mu$)}\\
&\approx \frac{2^{-N\alpha}}{N}\ \hspace{2.55cm}\text{ (by \eqref{want1-2})}\\
&\approx \frac{|I|^\alpha}{\log(1/|I|)}\ \hspace{1.85cm}\text{ (since $|I|\approx 2^{-N}$).}
\end{align*}
This finishes the proof.
\end{proof}

\begin{lemma}\label{lem:lower-regular-2}
For all intervals $I$ centered in $E$ with $|I|<1/2$,
$$\mu(I)\gtrsim \frac{|I|^\alpha}{\log(1/|I|)}.$$
\end{lemma}

\begin{proof}
It suffices to prove the estimate for small intervals. Assume that
$$\frac{2}{2^{N+1}}\le |I|<\frac{2}{2^N}.$$
Since $I$ is centered in $E_{N+1}$, there exists $a\in A_{N+1}$ such that $J=(0,1/2^{N+1})+a$ is contained in $I$. Now we can estimate
\begin{align*}
\mu(I)\ge\mu(J)
&={(t_1\cdots t_{N+1})^{-1}}\ \text{ (by the definition of $\mu$)}\\
&\approx\frac{2^{-N\alpha}}{N}\ \hspace{1.4cm}\text{ (by \eqref{want1-2})}\\
&\approx \frac{|I|^\alpha}{\log(1/|I|)}\ \hspace{0.7cm}\text{ (since $|I|\approx 2^{-N}$).}
\end{align*}
This finishes the proof.
\end{proof}

\section{Proof of Theorem \ref{thm:theorem-3}}\label{sec:theorem-3}

The construction of $E$ and $\mu$ for Theorem \ref{thm:theorem-3} differs from the one for Theorem \ref{thm:theorem-2}  only in the way we choose $t_N$. The corresponding Fourier decay and regularity estimates are proved by similar arguments. We sketch here only the main steps of the proof.

\subsection{Choosing $t_N$ and $N_0$}\label{sec:choosing-1}
As before, we write $t_N=2^{\alpha}\theta_N$. Instead of \eqref{eq:asympt-2}, we now want
\begin{align}\label{eq:asympt-3}
1\approx \theta_1\cdots\theta_N.
\end{align}
For the same reason as before, this can be guaranteed by choosing appropriate $t_N\in\{1,2\}$. Similarly, we then choose $N_0$ so large that \eqref{eq:bernstein-cond-2} is satisfied. Note that by our choice of $t_N$, instead of \eqref{want1-2} we now have
\begin{align}\label{want1-3}
{T_{N-1}^{-1/2}}\log^{1/2}(2^N 8)\approx \frac{N^{1/2}}{2^{N\alpha/2}}.
\end{align}

\subsection{Fourier decay estimate}\label{sec:fourier-decay-3}
By \eqref{eq:key-est-pre-2} and \eqref{want1-3}, we have
$$\Big|\frac{1}{T_{N-1}}\sum_{a\in A_{N-1}}\chi_a(k)\Big|\lesssim \frac{N^{1/2}}{2^{N\alpha/2}}.$$
By the same argument as in Section \ref{sec:fourier-decay-2}, to prove \eqref{eq:fourier-decay-3} it suffices to show
\begin{align*}
\sum_{N=1}^\infty\min\Big(1,\frac{2^N}{|k|}\Big)\frac{N^{1/2}}{2^{N\alpha/2}}
\lesssim \frac{\log^{1/2}|k|}{|k|^{\alpha/2}}
\end{align*}
for all $|k|\ge 2$. To this end, we write, as before,
\begin{align*}
\sum_{N=1}^\infty\min\Big(1,\frac{2^N}{|k|}\Big)\frac{N^{1/2}}{2^{N\alpha/2}}
&= \sum_{2^N\le |k|}\frac{2^N}{|k|}\frac{N^{1/2}}{2^{N\alpha/2}}
+\sum_{2^N>|k|}\frac{N^{1/2}}{2^{N\alpha/2}}\\
&= \frac{1}{|k|}\sum_{N\le N_1} N^{1/2}2^{N(1-\alpha/2)}+\sum_{N\ge N_2}\frac{N^{1/2}}{2^{N\alpha/2}}
\end{align*}
where $N_1$ and $N_2$ are as in Section \ref{sec:fourier-decay-2}. Since
$$\sum_{N\le N_1}N^{1/2} 2^{N(1-\alpha/2)}\lesssim N_1^{1/2} 2^{N_1(1-\alpha/2)}$$
and
$$\sum_{N\ge N_2}\frac{N^{1/2}}{2^{N\alpha/2}}\lesssim \frac{N_2^{1/2}}{2^{N_2\alpha/2}},$$
the desired estimate follows from the observation that
$$N_1^{1/2} 2^{N_1(1-\alpha/2)}\lesssim {|k|^{1-\alpha/2}\log^{1/2}|k|}$$
and
$$\frac{N_2^{1/2}}{2^{N_2\alpha/2}}\lesssim \frac{\log^{1/2}|k|}{|k|^{\alpha/2}}.$$

\subsection{Regularity estimate}
By the same argument as in Section \ref{sec:regularity-2}, the proof of \eqref{eq:regularity-3} reduces to the fact that
$$t_1\cdots t_N\approx t_1\cdots t_{N+1}\approx 2^{N\alpha},$$
which is an immediate consequence of \eqref{eq:asympt-3}.

\section{Proof of Theorem \ref{thm:theorem-1}}\label{sec:theorem-1}
To obtain the additional property \eqref{eq:sharpness} in Theorem \ref{thm:theorem-1}, we will follow the idea of Hambrook and \L aba \cite{hambrook-laba}, that is, to build a highly structured ``subtree'' $\{P_N\}$ into the ``random tree'' $\{A_N\}$ as constructed above. In order to show \eqref{eq:sharpness} in the indicated range of $q$, we need the ``branching number'' of $\{P_N\}$ to grow with $N$ overall. This requires the ``ambient tree'' to branch at growing rate and therefore results in a slight weakening the of regularity estimates as seen in \eqref{eq:regularity} and \eqref{eq:regularity-4}.

\subsection{Outline of the construction}
Given $0<\beta\le\alpha<1$ and $\phi$ as in Theorem \ref{thm:theorem-1}, we write
$$\psi(N)=\lceil\phi(2^{N})^{1/2}\rceil+2$$
and
$$\Psi(N)=\prod_{n=1}^N\psi(n).$$
Here $\psi(N)$ will indicate the ``branching number'' of the ``ambient tree'' as $2$ does in the previous construction.

Given two sequences $\{\tau_N\}_{N\ge1}$ and $\{t_N\}_{N\ge1}$ (to be chosen in Section \ref{sec:choosing-3}) with
$$1\le\tau_N\le t_N<\psi(N),$$
let $P_0,P_1,\cdots$ be the uniquely determined sequence of sets with $P_0=\{0\}$ and
$$P_N=\bigcup_{a\in P_{N-1}}\{1,\cdots,\tau_N\}/\Psi(N)+a.$$
We will construct inductively another sequence of sets $A_0, A_2, \cdots$ with
\begin{align*}
& \bullet\ \ A_0=\{0\}\\
& \bullet\ \ A_N=\bigcup_{a\in A_{N-1}}A_{N,a}+a\\
& \bullet\ \ A_{N,a}\subset [\psi(N)]/\Psi(N)\\
& \bullet\ \ \# A_{N,a}=t_N \\
& \bullet\ \ A_N\supset P_N.
\end{align*}
As before, we write
$$E_N=\bigcup_{a\in A_{N}}[0,1/\Psi(N)]+a$$
$$\mu_N=\frac{1}{|E_N|}\mathds{1}_{E_N}(t)dt$$
and let
$$E=\bigcap_{N=1}^\infty E_N$$
$$\mu=\lim_{N\rightarrow\infty} \mu_N.$$
In addition, we let
$$f_N=\mathds{1}_{F_N}$$
where
$$F_N=\bigcup_{a\in P_{N}}[0,1/\Psi(N))+a.$$

\subsection{The iteration}\label{sec:construction}
Given a large integer $N_0$ (to be chosen in Section \ref{sec:choosing-3}), we choose $A_1, \cdots,A_{N_0}$ arbitrarily as long as $E_{N_0}$ is compactly supported in $(0,1)$. Assuming that $A_1,\cdots, A_{N-1}$ have been chosen, we now choose $A_{N}$.

As before, let
$$B_N=\{0,1,\cdots,t_N-1\}/\Psi(N).$$
For $x\in [\psi(N)]$, let
$$B_{N,x}=\Big\{{y+x\ (\text{mod } \psi(N))}: y/\Psi(N)\in B_N\Big\}/{\Psi(N)}\subset [\psi(N)]/\Psi(N).$$
By the same argument as in Section \ref{sec:iteration}, we can find $x(a)$ for each $a\in A_{N-1}$ such that, writing
$$\chi_a(k)=\Big(\frac{S_{B_N,x(a)}(k)}{t_N}-\frac{S_{[\psi(N)]/\Psi(N)}(k)}{\psi(N)}\Big)e^{-2\pi iak}$$
and
$$T_{N}=\#A_{N}=t_1\cdots t_N,$$
the estimate
\begin{align}\label{eq:key-est-pre}
\Big|\frac{1}{T_{N-1}}\sum_{a\in A_{N-1}}\chi_a(k)\Big|
\le 4{T_{N-1}^{-1/2}}\log^{1/2}(8\Psi(N))
\end{align}
holds for all $k\in\mathbb Z$, provided that
\begin{equation}\label{eq:bernstein-cond}
4\log(8\Psi(N))\le T_{N-1}
\end{equation}
which will be guaranteed by our choice of $t_N$ and $N_0$ (see Section \ref{sec:choosing-3}).

The additional procedure now is to modify $B_{N,x(a)}$ in order to build in $P_N$ and make $P_N$ suitably isolated in $A_N$. More precisely, for $a\in P_{N-1}$, we adjoin $\{1,\cdots,\tau_N\}/\Psi(N)$ to $B_{N,x(a)}$ and subtract a matching number of elements of $B_{N,x(a)}$ that are not in $\{1,\cdots,\tau_N\}/\Psi(N)$, so that the resulting set still has cardinality $t_N$. If $0$ or $(\tau_N+1)/\Psi(N)$ is in this new set, we further modify it so that these two points are not included while the cardinality of the set is unchanged and $\{1,\cdots,\tau_N\}/\Psi(N)$ is still included. Let $A_{N,a}$ be the set obtained.

For $a\notin P_{N-1}$, we let $A_{N,a}=B_{N,x(a)}$. Now set
$$A_N=\bigcup_{a\in A_{N-1}}A_{N,a}+a.$$
This finishes our choice of $A_N$.

Note that, if we write
$$\tilde\chi_a(k)=\Big(\frac{S_{A_{N,a}}(k)}{t_N}-\frac{S_{[\psi(N)]/\Psi(N)}(k)}{\psi(N)}\Big)e^{-2\pi iak},$$
then by the modification procedure described above, we have $\tilde\chi_a(k)=\chi_a(k)$ for $a\notin P_{N-1}$, and
$$|\tilde\chi_a(k)-\chi_a(k)|\le \frac{2\tau_N+4}{t_N}\le \frac{6\tau_N}{t_N}$$
for $a\in P_{N-1}$. Therefore
$$\Big|\frac{1}{T_{N-1}}\sum_{a\in A_{N-1}}\tilde\chi_a(k)-\frac{1}{T_{N-1}}\sum_{a\in A_{N-1}}\chi_a(k)\Big|
\le \frac{\#P_{N-1}}{T_{N-1}}\frac{6\tau_N}{t_N}
= \frac{6\tau_1\cdots\tau_N}{t_1\cdots t_N}.$$
Combining this with \eqref{eq:key-est-pre}, we see that
\begin{align}\label{eq:key-est}
\Big|\frac{1}{T_{N-1}}\sum_{a\in A_{N-1}}\tilde\chi_a(k)\Big|
\le 4{T_{N-1}^{-1/2}}\log^{1/2}(8\Psi(N))+\frac{6\tau_1\cdots\tau_N}{t_1\cdots t_N}
\end{align}
holds for all $k\in\mathbb Z$.

Note also that, given $l>N_0$, since $P_l$ is isolated in $A_l$, we can find a smooth function $\chi_{_l}$ so that
\begin{align}\label{eq:smooth-cutoff}
f_{l}\mu_N=\chi_{_l}\mu_N
\end{align}
holds for all $N\ge l$.

\subsection{{Choosing $\tau_N, t_N$ and $N_0$}}\label{sec:choosing-3}
Similar as before, we write
\begin{align}\label{eq:t_N}
t_N=\psi(N)^{\alpha}\theta_N \ \ \text{\ and\ }\ \ \tau_N=\psi(N)^{\alpha-\beta/2}\vartheta_N.
\end{align}
For sufficiently large $N$, say $N>M$, we can (and will) choose $t_N$ so that
$$1/4\le\theta_N\le 1/2 \ \ \text{\ or\ }\ \ 2\le\theta_N\le 4,$$
and similarly for  $\tau_N$ and $\vartheta_N$. Such an $M$ does exist since $\lim_{N\rightarrow\infty}\psi(N)=\infty$. For $N\le M$, we let $t_N=\tau_N=1$. We now choose $t_N$ and $\tau_N$ for $N>M$. From now on all implied constants depend only on $\alpha,\beta$ and $\phi$.

In view of \eqref{eq:key-est}, we want
\begin{align}\label{want1}
{T_{N-1}^{-1/2}}\log^{1/2}(8\Psi(N))\approx {\Psi(N)^{-\alpha/2}}
\end{align}
to hold for all $N>N_0$, which reduces to
\begin{align}\label{want1-11}
\psi(N)^{\alpha}\log(8\Psi(N))\approx \theta_1\cdots\theta_{N-1}.
\end{align}
Since the left hand side of \eqref{want1-11} grows mildly, we can indeed find suitable $\theta_N$ for $N>M$ so that \eqref{want1-11} holds for all $N$. This finishes our choice of $t_N$.

In view of \eqref{eq:key-est}, we also want
\begin{align}\label{want2}
\frac{\tau_1\cdots\tau_N}{t_1\cdots t_N}
\approx {\Psi(N)^{-\beta/2}}
\end{align}
to hold for all $N>N_0$, which reduces to
\begin{align}\label{eq:equiv-1}
\vartheta_1\cdots\vartheta_{N}
\approx \theta_1\cdots\theta_N.
\end{align}
By our choice of $t_N$, this is true if
\begin{align}\label{eq:equiv-2}
\vartheta_1\cdots\vartheta_{N}\approx\psi(N+1)^{\alpha}\log(8\Psi(N+1)).
\end{align}
But this can be guaranteed by choosing appropriate $\vartheta_N$ for $N>M$. This finishes our choice of $\tau_N$. Now one can clearly choose $N_0$ large enough so that \eqref{eq:bernstein-cond} is ensured for all $N>N_0$.

\subsection{Fourier decay estimate}\label{sec:fourier-decay-1}
Since $\beta\le\alpha$, by \eqref{eq:key-est}, \eqref{want1} and \eqref{want2},
$$\Big|\frac{1}{T_{N-1}}\sum_{a\in A_{N-1}}\tilde\chi_a(k)\Big|\lesssim {\Psi(N)^{-\beta/2}}$$
holds for all $N>N_0, k\in\mathbb Z$. Now that
\begin{align*}
|\hat\mu_N(k)-\hat\mu_{N-1}(k)|
&= \Big|\mathcal F\big(\Psi(N)\mathds{1}_{[0,1/\Psi(N))}\big)(k)\Big|\Big|\frac{1}{T_{N-1}}\sum_{a\in A_{N-1}}\tilde\chi_a(k)\Big|\\
&\lesssim \min\Big(1,\frac{\Psi(N)}{|k|}\Big){\Psi(N)^{-\beta/2}}
\end{align*}
for all $N>N_0$, $k\neq 0$, arguing as in Section \ref{sec:fourier-decay-2}, we see that
\begin{align}\label{fourier-decay-N}
|\hat\mu_N(k)|\lesssim {|k|^{-\beta/2}}
\end{align}
for all $N>N_0$, $k\neq 0$. This implies \eqref{eq:fourier-decay} for the same reason as before.

Note that, by \eqref{eq:smooth-cutoff} and \eqref{fourier-decay-N}, we also have
\begin{align}\label{dominating-function}
|\widehat{f_l\mu_N}(\xi)|
&=|\widehat{\chi_{_l}\mu_N}(\xi)|
\lesssim_{_l} {|\xi|^{-\beta/2}}
\end{align}
for all $N_0<l<N$ and $\xi\in\mathbb R, \xi\neq 0$.

\subsection{Regularity estimate}\label{sec:regularity}
As before, we split the proof of \eqref{eq:regularity} into two parts.

\begin{lemma}\label{regular-1}
For all intervals $I$ with $|I|<1/2$,
$$\mu(I)\lesssim \frac{|I|^\alpha}{\log(1/|I|)}.$$
\end{lemma}

\begin{proof}
Assume that
$$\frac{1}{\Psi(N+1)}\le |I|<\frac{1}{\Psi(N)}$$
and that there exists $a\in A_{N}$ such that the interval $(0,1/\Psi(N))+a$ intersects with $I$. Since $|I|<1/\Psi(N)$, there are at most two such intervals, say $J_1$ and $J_2$. Now we can estimate
\begin{align*}
\mu(I)
&=\mu(I\cap J_1)+\mu(I\cap J_2)\  \hspace{0.1cm}\text{ (since $\mu$ is supported on $E_N$)}\\
&\le \mu(J_1)+\mu(J_2)\\
&={2}T_N^{-1}\  \hspace{2.8cm}\text{ (by the definition of $\mu$)}\\
&\approx \frac{\Psi(N+1)^{-\alpha}}{\log(8\Psi(N+1))}\ \hspace{1cm}\text{ (by \eqref{want1})}\\
&\le \frac{|I|^\alpha}{\log(1/|I|)}\ \hspace{2.0cm}\text{ (since $1/\Psi(N+1)\le |I|$).}
\end{align*}
This finishes the proof.
\end{proof}

\begin{lemma}\label{regular-2}
For all intervals $I$ centered in $E$ with $|I|<1/2$,
$$\mu(I)\gtrsim\frac{|I|^\alpha}{\phi(1/|I|)\log(1/|I|)}.$$
\end{lemma}

\begin{proof}
Assume that
$$\frac{2}{\Psi(N+1)}\le |I|<\frac{2}{\Psi(N)}$$
for some large $N$. Since $I$ is centered in $E_{N+1}$, there exists $a\in A_{N+1}$ such that $J=(0,1/\Psi(N+1))+a$ is contained in $I$. Now we can estimate
\begin{align*}
\mu(I)
&\ge \mu(J)\\
&=T_{N+1}^{-1}\ \hspace{4.1cm}\text{ (by the definition of $\mu$)}\\
&=T_{N-1}^{-1}\frac{1}{t_N t_{N+1}}\\
&\approx \frac{\Psi(N)^{-\alpha}}{\log(8\Psi(N))} \frac{1}{\psi(N)^\alpha \psi(N+1)^\alpha}\ \text{ (by \eqref{want1} and \eqref{eq:t_N})}\\
&\gtrsim \frac{\Psi(N)^{-\alpha}}{\log(\Psi(N)/2)} \frac{1}{\psi(N)^2}\ \hspace{1.65cm}\text{ (since $\alpha\le 1$)}\\
&\gtrsim \frac{|I|^{\alpha}}{\log(1/|I|)} \frac{1}{\phi(2^N)}\ \hspace{2.27cm}\text{ (since $|I|<2/\Psi(N)$)}\\
&\gtrsim \frac{|I|^{\alpha}}{\log(1/|I|)} \frac{1}{\phi(\Psi(N)/2)}\ \hspace{1.4cm}\text{ (since $2^N\le\Psi(N)$)}\\
&\gtrsim \frac{|I|^{\alpha}}{\log(1/|I|)} \frac{1}{\phi(1/|I|)}\ \hspace{1.95cm}\text{ (since $\phi$ is increasing)}.
\end{align*}
This finishes the proof.
\end{proof}

\subsection{Arithmetic progressions}
In this section we prepare for the proof of \eqref{eq:sharpness} following the arguments in \cite{hambrook-laba}. Roughly speaking, the idea is that the generalized arithmetic progressions $P_{l}$ embedded in $\mu$ give rise to large additive energy and therefore large Fourier transform of $f_l \mu$ in $L^q$.

More precisely, given $N_0<l<N$ and $r\in\mathbb Z_{\ge 1}$, we first deduce a lower bound for the additive energy
$$M_{l,N,r}=\#\Big\{(a_1,\cdots,a_{2r})\in \big(F_l\cap A_N\big)^{2r}: \sum_{j=1}^r a_j=\sum_{j=r+1}^{2r} a_j\Big\}.$$
By the Cauchy-Schwarz inequality, such a lower bound can be deduced from an upper bound on the size of the sum set
$$Z=\{a_1+\cdots+a_r: a_1,\cdots,a_r\in F_l\cap A_N\}.$$
To bound $\#Z$, notice that each $y\in F_l\cap A_N$ has a digit representation
$$y=\sum_{k=1}^l y^{(k)}/\Psi(k)+y^{(l+1)}/\Psi(N)$$
where $y^{(k)}\in \{1,\cdots,\tau_k\}$ for $k\le l$ and $y^{(l+1)}\in [\psi(l+1)\cdots \psi(N)]$. Thus $a_1+\cdots+a_r$ is of the form
$$\sum_{k=1}^l z^{(k)}/\Psi(k)+z^{(l+1)}/\Psi(N)$$
where $z^{(k)}\in \{r,r+1,\cdots, r\tau_k\}$ for $k\le l$ and
$z^{(l+1)}\in [r\psi(l+1)\cdots \psi(N)]$. From this we see that
\begin{align}\label{z-bound}
\#Z\le (r\tau_1)\cdots(r\tau_l)(r\psi(l+1)\cdots \psi(N))=r^{l+1}\tau_1\cdots\tau_l\frac{\Psi(N)}{\Psi(l)}.
\end{align}

For $z\in \mathbb Z/\Psi(N)$, if we write
$$g(z)=\#\{(y_1,\cdots,y_r)\in (F_l\cap A_N\big)^{r}, y_1+\cdots+y_r=z\},$$
then it is easy to see that
$$\|g\|_{\ell^1}=(\#F_l\cap A_N\big)^{r}=(\tau_1\cdots\tau_lt_{l+1}\cdots t_N)^r$$
and
$$\|g\|_{\ell^2}^2=M_{l,N,r}.$$
Combining these with $\|g\|_{\ell^1}^2\le \#Z\cdot\|g\|_{\ell^2}^2$ and \eqref{z-bound}, we obtain
\begin{align}\label{convolution}
M_{l,N,r}
\ge \frac{\|g\|_{\ell^1}^2}{\#Z}
\ge \frac{(\tau_1\cdots\tau_lt_{l+1}\cdots t_N)^{2r}{\Psi(l)}}{r^{l+1}\tau_1\cdots\tau_l{\Psi(N)}},
\end{align}
as desired.

We then use Salem's trick\footnote{One may also use the Plancherel theorem to obtain a similar lower bound.} to bound the $L^{2r}$ norm of $\widehat{f_l\mu_N}$ from below. To do this, notice that
$$f_l\mu_N=\frac{1}{T_N}\big(\sum_{a\in F_l\cap A_N}\delta_a\big)*\Psi(N)\mathds{1}_{[0,1/\Psi(N))}.$$
So, using basic properties of the Fourier transform, we have
\begin{align*}
\widehat{f_l\mu_N}(\xi)
&= e^{-\pi i\xi/\Psi(N)}\frac{1}{T_N}\Big(\sum_{a\in F_l\cap A_N}e^{-2\pi ia\xi}\Big)\mathcal F\Big(\mathds{1}_{[-1/2,1/2)}\Big)\big(\xi/\Psi(N)\big).
\end{align*}
Using the formula $|z|^2=z\bar z$, we can then write
\begin{align*}
\big|\widehat{f_l\mu_N}(\xi)\big|^{2r}
&= \frac{1}{T_N^{2r}}\Big|\sum_{a\in F_l\cap A_N}e^{-2\pi ia\xi}\Big|^{2r}
\Big[\mathcal F\big(\mathds{1}_{[-1/2,1/2)}\big)(\xi/\Psi(N))\Big]^{2r}\\
&= \frac{1}{T_N^{2r}}\Big(\sum_{a_1,\cdots,a_{2r}\atop\in F_l\cap A_N}e^{-2\pi i(\sum_{j=1}^r a_j-\sum_{j=r+1}^{2r} a_j)\xi}\Big)\mathcal F\big(\mathds{1}_{[-1/2,1/2)}^{*2r}\big)\big(\xi/\Psi(N)\big).
\end{align*}
Thus, by the Fourier inversion theorem and the positivity of $\mathds{1}_{[-1/2,1/2)}^{*2r}$,
\begin{align*}
&\ \ \ \ \Big\|\widehat{f_l\mu_N}\Big\|^{2r}_{L^{2r}(\mathbb R)}\\
&= \frac{1}{T_N^{2r}} \int_{\mathbb R} \sum_{a_1,\cdots,a_{2r}\atop\in F_l\cap A_N} e^{-2\pi i(\sum_{j=1}^r a_j-\sum_{j=r+1}^{2r} a_j)\xi}\mathcal F\big(\mathds{1}_{[-1/2,1/2)}^{*2r}\big)\big(\xi/\Psi(N)\big)d\xi\\
&= \frac{\Psi(N)}{T_N^{2r}} \sum_{a_1,\cdots,a_{2r}\atop\in F_l\cap A_N} \int_{\mathbb R} e^{-2\pi i(\sum_{i=j}^r a_j-\sum_{j=r+1}^{2r} a_j)\Psi(N)\xi}\mathcal F\big(\mathds{1}_{[-1/2,1/2)}^{*2r}\big)(\xi)d\xi\\
&= \frac{\Psi(N)}{T_N^{2r}} \sum_{a_1,\cdots,a_{2r}\atop\in F_l\cap A_N} \mathds{1}_{[-1/2,1/2)}^{*2r}\Big(\big(\sum_{j=1}^r a_j-\sum_{j=r+1}^{2r} a_j\big)\Psi(N)\Big)\\
&\ge \frac{\Psi(N)}{T_N^{2r}} M_{l,N,r} \mathds{1}_{[-1/2,1/2)}^{*2r}(0).
\end{align*}
Combining this with \eqref{convolution}, we obtain
\begin{align}
\Big\|\widehat{f_l\mu_N}\Big\|^{2r}_{L^{2r}(\mathbb R)}\notag
&\ge C_r\frac{\Psi(N)}{T_N^{2r}}\frac{(\tau_1\cdots\tau_lt_{l+1}\cdots t_N)^{2r}{\Psi(l)}}{r^{l+1}\tau_1\cdots\tau_l{\Psi(N)}}\\
&= C_r \frac{\Psi(l)(\tau_1\cdots\tau_l)^{2r-1}}{r^{l+1}T_l^{2r}}\label{norm-below}
\end{align}
where $C_r>0$ is a constant depending only on $r$.

\subsection{Failure of the extension estimate for $q<2+\frac{4(1-\alpha)}{\beta}$}\label{sec:failure}
We are now ready to prove \eqref{eq:sharpness}. Arguing as in \cite{hambrook-laba}, we will fix $r\in\mathbb Z$ with
$$r>\frac{1}{\beta}\ \ \text{ and }\ \ 2r>2+\frac{4(1-\alpha)}{\beta}.$$
First we show that the lower bound \eqref{norm-below} extends to the limiting measure $\mu$. To see this, notice that since $f_l\mu_N$ converges to $f_l\mu$ weakly, we have
$$\lim_{N\rightarrow\infty}\widehat{f_l\mu_N}(\xi)=\widehat{f_l\mu}(\xi)$$
for all $\xi\in\mathbb R$. On the other hand, by \eqref{dominating-function} and our choice of $r$, $|\widehat{f_l\mu_N}(\xi)|^{2r}$ is dominated by an integrable function independent of $N>l$. So, by the dominated convergence theorem, we can let $N\rightarrow\infty$ in \eqref{norm-below} and conclude
\begin{align}\label{eq:lower}
\Big\|\widehat{f_l\mu}\Big\|^{2r}_{L^{2r}(\mathbb R)}
\ge C_r \frac{\Psi(l)(\tau_1\cdots\tau_l)^{2r-1}}{r^{l+1}T_l^{2r}}
\end{align}
for all $l>N_0$, as desired.

Next, we turn this lower bound into a lower bound for $\|\widehat{f_l\mu}\|_{L^{q}}$. To do this, suppose $2\le q<2r$. Then by writing $|z|^{2r}=|z|^{q}|z|^{2r-q}$, we have
\begin{align*}
\Big\|\widehat{f_l\mu}\Big\|^{2r}_{L^{2r}(\mathbb R)}
&\le \Big\|\widehat{f_l\mu}\Big\|^{q}_{L^{q}(\mathbb R)}
\Big\|\widehat{f_l\mu}\Big\|^{2r-q}_{L^\infty(\mathbb R)}\\
&\le \Big\|\widehat{f_l\mu}\Big\|^{q}_{L^{q}(\mathbb R)}
\big\|f_l\mu\big\|^{2r-q}_{\mathcal M(\mathbb R)}\\
&= \Big\|\widehat{f_l\mu}\Big\|^{q}_{L^{q}(\mathbb R)}
\Big(\frac{\tau_1\cdots\tau_l}{t_1\cdots t_l}\Big)^{2r-q}.
\end{align*}
Combining this with \eqref{eq:lower}, we get
\begin{align}
\Big\|\widehat{f_l\mu}\Big\|^{q}_{L^{q}(\mathbb R)}\notag
&\ge \Big\|\widehat{f_l\mu}\Big\|^{2r}_{L^{2r}(\mathbb R)}
\Big(\frac{t_1\cdots t_l}{\tau_1\cdots\tau_l}\Big)^{2r-q}\\
&\ge C_r \frac{\Psi(l)(\tau_1\cdots\tau_l)^{q-1}}{r^{l+1}T_l^{q}}\label{norm-below-q}
\end{align}
for all $l>N_0$, as desired.

Now suppose for contradiction that there exists a constant $C_q$ such that
\begin{align}\label{contradiction}
\Big\|\widehat{f_l\mu}\Big\|^{q}_{L^{q}(\mathbb R)}
\le C_q\|f_l\|_{L^2(\mu)}^q
\end{align}
holds for all $l>N_0$. Then from \eqref{norm-below-q}, we deduce
\begin{align*}
C_r \frac{\Psi(l)(\tau_1\cdots\tau_l)^{q-1}}{r^{l+1}T_l^{q}}
&\le C_q\|f_l\|_{L^2(\mu)}^q\\
&= C_q\Big(\frac{\tau_1\cdots \tau_l}{t_1\cdots t_l}\Big)^{q/2}.
\end{align*}
Using the notation in Section \ref{sec:choosing-3}, this means
\begin{align*}
C_{q,r}{\Psi(l)}^{1-\alpha q/2+(\alpha-\beta/2)(q/2-1)}
\le {r^{l+1}}\frac{(\theta_1\cdots \theta_l)^{q/2}}{(\vartheta_1\cdots\vartheta_l)^{q/2}}(\vartheta_1\cdots\vartheta_l).
\end{align*}
By \eqref{eq:equiv-1} and \eqref{eq:equiv-2}, this implies
\begin{align}\label{eq:contradiction}
C_{q,r}{\Psi(l)}^{1-\alpha q/2+(\alpha-\beta/2)(q/2-1)}
&\lesssim r^{l+1} \psi(l+1)^{\alpha}\log\big(8\Psi(l+1)\big).
\end{align}
If $2\le q<2+\frac{4(1-\alpha)}{\beta}$, then
$$1-\alpha q/2+(\alpha-\beta/2)(q/2-1)>0.$$
So the left hand side of \eqref{eq:contradiction} grows faster than any geometric growth as $l\rightarrow\infty$ (since $\lim_{N\rightarrow\infty}\psi(N)=\infty$) whereas the right hand side of \eqref{eq:contradiction} is bounded by a geometric growth in $l$. This gives a contradiction to \eqref{contradiction}, and therefore shows that, for $2\le q<2+\frac{4(1-\alpha)}{\beta}$, we have
$$\sup_{l\ge 1}\frac{\|\widehat{f_l\mu}\big\|_{L^{q}(\mathbb R)}}{\|f_l\|_{L^2(\mu)}}=\infty.$$
This completes the proof of the first part of Theorem \ref{thm:theorem-1}.

\subsection{Proof of the second part}
We now turn to the second part of Theorem \ref{thm:theorem-1}. Similar as in the proof of Theorem \ref{thm:theorem-3}, the desired properties will follow from adjusting $\tau_N$ and $t_N$ in Section \ref{sec:choosing-3} for $N>M$.

More precisely, instead of \eqref{want1} we now want
\begin{align}\label{want3}
{T_{N-1}^{-1/2}}\approx {\Psi(N)^{-\alpha/2}}
\end{align}
to hold for all $N>N_0$, which reduces to
$$\psi(N)^{\alpha}\approx \theta_1\cdots\theta_{N-1}.$$
For the same reason as before, this can be guaranteed for all $N$ by choosing appropriate $\theta_N$ for $N>M$. This finishes our choice of $t_N$.

We will also want \eqref{want2} to hold for all $N>N_0$. As before, this reduces to \eqref{eq:equiv-1}, and therefore to
\begin{align}\label{eq:equiv-3}
\vartheta_1\cdots\vartheta_{N}\approx\psi(N+1)^{\alpha}.
\end{align}
But this can be guaranteed by choosing appropriate $\vartheta_N$ for $N>M$. This finishes our choice of $\tau_N$.

To show the Fourier decay property \eqref{eq:fourier-decay}, notice that since
$${\Psi(N)^{-\alpha/2}}{\log^{1/2}(8\Psi(N))}\lesssim {\Psi(N)^{-\beta/2}},$$
by \eqref{eq:key-est}, \eqref{want3} and \eqref{want2} we still have
$$\Big|\frac{1}{T_{N-1}}\sum_{a\in A_{N-1}}\tilde\chi_a(k)\Big|\lesssim {\Psi(N)^{-\beta/2}}$$
for all $N>N_0, k\in\mathbb Z$. \eqref{eq:fourier-decay} then follows from the same argument as in Section \ref{sec:fourier-decay-1}.

The second part of the regularity property \eqref{eq:regularity-4} follows from the proof of Lemma \ref{regular-1}, except that we now estimate, by \eqref{want3},
$$T_N^{-1}\approx \Psi(N+1)^{-\alpha}\lesssim |I|^{\alpha}.$$
The first part of \eqref{eq:regularity-4} follows from the proof of Lemma \ref{regular-2}, except that we now estimate, again by \eqref{want3},
$$T_{N-1}^{-1}\approx \Psi(N)^{-\alpha}\gtrsim {|I|^\alpha}.$$

Finally, the proof of \eqref{eq:sharpness} goes in the same way as in Section \ref{sec:failure}, except that we now arrive at, instead of \eqref{eq:contradiction},
\begin{align}\label{eq:contradiction-2}
C_{q,r}{\Psi(l)}^{1-\alpha q/2+(\alpha-\beta/2)(q/2-1)}
&\lesssim r^{l+1} \psi(l+1)^{\alpha}.
\end{align}
Here we have used \eqref{eq:equiv-1} and \eqref{eq:equiv-3}. If $2\le q< 2+\frac{4(1-\alpha)}{\beta}$, then \eqref{eq:contradiction-2} gives a contradiction to \eqref{contradiction} for the same reason as before. This shows \eqref{eq:sharpness}, and the proof of Theorem \ref{thm:theorem-1} is complete.

\subsection{Remarks} (1) The restriction $t\le 1$ in Lemma \ref{lem:bernstein} can be removed. Consequently, the introduction of the large integer $N_0$ can likely be eliminated. We thank Kyle Hambrook for pointing this out.

(2) If $\mu$ is as in the second part of Theorem \ref{thm:theorem-1} where $\beta<\alpha$, then the exponent $\beta$ in \eqref{eq:fourier-decay} is optimal, since otherwise \eqref{eq:sharpness} would be violated according to \eqref{eq:regularity-4} and the $L^2$-Fourier restriction theorem.

(3) Instead of using the trick in Section \ref{sec:choosing-3}, one may want to employ an $M$-out-of-$N$ construction with growing $N$ and suitably rounded $M$. But this does not work out efficiently since the rounding errors can accumulate into a big scale, as exemplified by the following elementary estimate:
$$e^{C_1 L^{1-\alpha}}\le \prod_{N=1}^L\frac{\lceil N^\alpha\rceil}{N^\alpha}\le e^{C_2 L^{1-\alpha}}.$$

\section{Acknowledgement}
The author would like to thank Andreas Seeger for many helpful suggestions that improved the content and presentation of this paper.

\end{document}